\documentclass{article}
\usepackage{amsmath,amsthm,amscd,amssymb}
\usepackage[mathscr]{eucal}
\usepackage{amssymb}
\usepackage{latexsym}

\theoremstyle{definition}
\newtheorem{Def}{Definition}[section]
\newtheorem{Thm}[Def]{Theorem}

\newtheorem{Rem}[Def]{Remark}
\newtheorem{Ex}[Def]{Example}

\newtheorem{Lem}[Def]{Lemma}

\numberwithin{equation}{section}

\newcommand{\spl}{\mathrm{Sp}}
\newcommand{\Z}{\mathbb{Z}}
\newcommand{\sll}{\mathrm{SL}}

\title{On the mod $p$ kernel of the theta operator and Eisenstein series}

\author{Shoyu Nagaoka and Sho Takemori}

\date{}
\begin{document}

\maketitle

%\keywords{}
%\noindent
%{\bf Mathematics subject classification 2010}: Primary 11F33, \and Secondary 11F46\\
%\noindent
%{\bf Key words}: Siegel modular forms, Congruences for modular forms, Leech lattice
%
%%%%%%%%%%ABSTRACT%%%%%%%%%%%%%%%
\begin{abstract}
\noindent
Siegel modular forms in the space of the mod $p$ kernel of the theta operator
are constructed by the Eisenstein series in some odd-degree cases.
Additionally, a similar result in the case of Hermitian modular forms is given.
\end{abstract}
%%%%%%%%%%%%%%%%%%%%%%%%%%%%%%%%%
%\maketitle

%\noindent
%{\bf Mathematics subject classification}: Primary 11F33 \and Secondary 11F46\\
%\noindent
%{\bf Key words}:
%%%%%%%%%%INTRODUCTION%%%%%%%%%%%%
\section{Introduction}
\label{intro}

The theta operator is a kind of differential operator operating on modular forms.
Let $F$ be a Siegel modular form with the generalized $q$-expansion
$F=\sum a(T)q^T$, $q^T:=\text{exp}(2\pi i\text{tr}(TZ)))$.
The theta operator $\varTheta$ is defined as
$$
\varTheta : F=\sum a(T)q^T
\longmapsto
\varTheta (F):=\sum a(T)\cdot\text{det}(T)q^T,
$$
which is a generalization of the classical Ramanujan's $\theta$-operator.
It is known that the notion of singular modular form $F$ is characterized by
$\varTheta (F)=0$.

For a prime number $p$, the mod $p$ kernel of the theta operator is defined
as the set of modular form $F$ such that $\varTheta (F) \equiv 0 \pmod{p}$.
Namely, the element in the kernel of the theta operator can be interpreted as
a mod $p$ analogue of the singular modular form.

In the case of Siegel modular forms of even degree, several examples are known
(cf.  Remark \ref{exampleS}).
In \cite{Na}, the first author constructed such a form by using Siegel
Eisenstein series in the case of even degree. However little is known about
the existence of such a modular form in the case of odd degree.

In this paper, we shall show that some
odd-degree Siegel Eisenstein series give examples of modular forms in the mod $p$
kernel of the theta operator (see Theorem \ref{main1}). Our proof is based on Katsurada's functional equation of
Kitaoka's polynomial appearing as the main factor of the Siegel series.

For a Siegel modular form $F \in M_{k}(\spl_{n}(\Z))_{\Z_{(p)}}$ (here the
subscript $\Z_{(p)}$ means every Fourier coefficient of $F$ belongs to $\Z_{(p)}$),
we denote by $\omega(F)$ the filtration of $F \bmod{p}$, that is, minimum
weight $l$ such that there exists $G \in M_{l}(\spl_{n}(\Z))_{\Z_{(p)}}$ and
$F \equiv G \mod{p}$ (congruence between $q$-expansions).  Our ultimate aim is
that, for a given weight $k$, list all $F \in M_{k}(\spl_{n}(\Z))_{\Z_{(p)}}$
and $\Theta(F) \equiv 0 \mod{p}$ such that $\omega(F) = k$.  In elliptic
modular form case, this problem was already solved (cf.
\cite{Se}, \cite{Kat}) and there
is a simple description. Assume $p \ge 5$ and let
$f \in M_{k}(\sll_{2}(\Z))_{\Z_{(p)}}$ with $\Theta(f) \equiv 0 \mod{p}$ and
$\omega(f) = k$, then $k$ is divisible by $p$ and there exists
$g \in M_{k/p}(\sll_{2}(\Z))_{\Z_{(p)}}$ such that
$f \equiv g^{p} \mod{p}$ and $\omega(g) = k/p$.

There are a several methods to construct $F \in M_{k}(\spl_{n}(\Z))_{\Z_{(p)}}$ with
$\Theta(F) \equiv 0 \mod{p}$ other than by using Eisenstein series.
\begin{enumerate}
  \item By theta series (with harmonic polynomials) associated to quadratic forms
  with discriminant divisible by $p$.
  \item By the operator $A(p)$.
\end{enumerate}
B\"{o}cherer, Kodama and the first author argue the first method in \cite{Bo-Ko-Na}.
In several cases, it gives $F \in M_{k}(\spl_{n}(\Z))_{\Z_{(p)}}$ with $\Theta(F) \equiv 0 \mod{p}$
and $\omega(F) = k$.
As for the second method, the operator $A(p)$ is defined by
\begin{equation*}
  F|A(p) \equiv F - \Theta^{(p-1)}F \mod{p}.
\end{equation*}
This operator was introduced in \cite{C-C-R} and \cite{De-Ri}.  If
$\Theta(F) \equiv 0 \mod{p}$, then we have $F|A(p) \equiv F \mod{p}$.
Therefore for any
$F \in M_{k}(\spl_{n}(\Z))_{\Z_{(p)}}$ with $\Theta(F) \equiv 0 \mod{p}$,
there exists $l \in \Z_{\ge 0}$ and $G \in M_{l}(\spl_{n}(\Z))_{\Z_{(p)}}$ such that
$F \equiv G|A(p) \mod{p}$.
However it seems difficult to compute $\omega(F|A(p))$ in terms of
$\omega(F)$, and the filtration $\omega(F|A(p))$ can be large compared to
$\omega(F)$ (cf. \cite[\S 4, \S 6]{Bo-Ki-Ta}).

Additionally, we give a similar result in the case of
Hermitian modular forms (Theorem \ref{main2}). In this case, we use Ikeda's functional equation which is the
corresponding result of Katsurada's one.

%%%%%%%%%%%%%%%%%%%%%%%%%%%%%%%%%%%%%%%%%%%%%%%%
%%%%%%%%%%%%%%%%%%%%%%%%%%%%%%%%%%%%%%%%%%%%%%%%%%%%%%%%%%%%%%%%%%%%%
\section{Siegel modular case}
\subsection{Siegel modular forms}
Let $\Gamma^{(n)}=\text{Sp}_n(\mathbb{Z})$ be the Siegel modular group of degree $n$
and $M_k(\Gamma^{(n)})$ be the space of Siegel modular forms of weight $k$ for
$\Gamma^{(n)}$. Any element $F$ in $M_k(\Gamma^{(n)})$ has a Fourier expansion
of the form
$$
F(Z)=\sum_{0\leq T\in\Lambda_n}a(T;F)q^T,\quad
q^T:=\text{exp}(2\pi i\text{tr}(TZ)),\quad Z\in\mathbb{H}_n,
$$
where
\begin{align*}
& \mathbb{H}_n=\{\,Z\in\text{Sym}_n(\mathbb{C})\,\mid\, \text{Im}(Z)>0\,\}\;\;
(\text{the Siegel upper half space}),\\
& \Lambda_n:=\{\,T=(t_{jl})\in\text{Sym}_n(\mathbb{Q})\,\mid\,
t_{jj}\in\mathbb{Z},\,2t_{jl}\in\mathbb{Z}\,\}.
\end{align*}
We also denote by $S_k(\Gamma^{(n)})$
the space of $M_k(\Gamma^{(n)})$ consisting of cusp forms.

For a subring $R\subset \mathbb{C}$, $M_k(\Gamma^{(n)})_R$
(resp. $S_k(\Gamma^{(n)})_R$) consists of an element  $F$ in $M_k(\Gamma^{(n)})$
(resp. $S_k(\Gamma^{(n)})$) whose Fourier coefficients $a(T;F)$ lie in $R$.
%%%%%%%%%%%%
%%%%%%%%%%%%%%
\subsection{Theta operator}
For an element $F$ in  $M_k(\Gamma^{(n)})$, we define
$$
\varTheta : F=\sum a(F;T)q^T\,\longmapsto\,
\varTheta (F):=\sum a(F;T)\cdot \text{det}(T)q^T
$$
and call it the {\it theta operator}. It should be noted that $\varTheta (F)$ is
not necessarily a Siegel modular form.
However, we have the following result.
%%%%%%%%%%%%%%
\begin{Thm} {\rm (B\"{o}cherer-Nagaoka \cite{B-N})}
Let $p$ be a prime number with $p \ge n+3$ and $\mathbb{Z}_{(p)}$ be the ring of
$p$-integral rational numbers.  If $F\in M_k(\Gamma^{(n)})_{\mathbb{Z}_{(p)}}$,
then there exists a cusp form $G\in S_{k+p+1}(\Gamma^{(n)})_{\mathbb{Z}_{(p)}}$
such that
$$
\varTheta (F) \equiv G \pmod{p},
$$
where the congruence means the Fourier coefficient-wise one.
\end{Thm}
%%%%%%%%%%%%%%%
In some cases, it happens that $G \equiv 0 \pmod{p}$, namely,
$$
\varTheta (F) \equiv 0 \pmod{p}.
$$
In such a case, we say that the modular form $F$ is an element of the
{\it mod p kernel of the theta operator} $\varTheta$.

A Siegel modular form $F$ with $p$-integral Fourier coefficients is
called {\it mod $p$ singular} if it satisfies
$$
a(T;F) \equiv 0 \pmod{p}
$$
for all $T\in\Lambda_n$ with $T>0$.
Of course, a mod $p$ singular modular form $F$ satisfies
$\varTheta (F) \equiv 0 \pmod{p}$.

If an element $F$ of the mod $p$ kernel of the theta operator is not
mod $p$ singular, we call it here {\it essential}.

The main purpose of this paper is to construct essential forms by using
Eisenstein series.
%%%%%%%%%%%%%%%%%%%%%%%%%%%%%%%%%%%%%%%%%%%%%%%%%%%%%%%
\subsection{Siegel Eisenstein series}
Let
$$
\Gamma_\infty^{(n)}:=\left\{ \begin{pmatrix} A& B \\ C & D \end{pmatrix}
\in \Gamma^{(n)}\,\Big{|}\,C=0_n\right\}.
$$
For an even integer $k>n+1$, the {\it Siegel Eisenstein series of weight $k$}
is defined by
$$
E_k^{(n)}(Z):=\sum_{\binom{*\;*}{C\,D}\in \Gamma_\infty^{(n)}\backslash \Gamma^{(n)}}
\text{det}(CZ+D)^{-k}.
$$
We set $\Lambda_n^+=\{  \,T\in\Lambda_n \mid T>0\,\}$. For $T\in\Lambda_n^+$,
we define $D(T)=2^{2[n/2]}\text{det}(T)$ and, if $n$ is even, then $\chi_T$ denotes the
primitive Dirichlet character corresponding to the extension
$K_{T} = \mathbb{Q}(\sqrt{(-1)^{n/2}\text{det}(2T)}\,)/\mathbb{Q}$.
We define a positive integer $C(T)$ by
\begin{equation*}
  C(T) =
  \begin{cases}
    D(T)/\mathfrak{d}_{T} & n: \text{even},\\
    D(T) & n: \text{odd}.
  \end{cases}
\end{equation*}
Here $\mathfrak{d}_{T}$ is the absolute value of the discriminant of $K_{T}/\mathbb{Q}$.

It is known that
the Fourier coefficient $a(T;E_k^{(n)})$\,$(T\in\Lambda_n^+)$ can be expressed as follows
(cf. \cite{Sh}, \cite{Sh-book}, \cite{Ka}, \cite{Ikeda-Siegel}, and \cite{Take}).
\begin{equation}
\begin{split}
\label{explicit}
a(T;E_k^{(n)})=& \zeta(1-k)^{-1}\prod_{i=1}^{[\frac{n}{2}]}\zeta(1+2i-2k)^{-1}\cdot
                  \prod_{\substack{q\mid C(T)\\q:\text{prime}}}F_q(T,q^{k-n-1})\\
                  & \times \begin{cases}
                     2^{n/2}\,L(1+\tfrac{n}{2}-k;\chi_T)  &\text{($n$: even)}\\
                     2^{(n+1)/2}                             &\text{($n$: odd)},
                     \end{cases}
\end{split}
\end{equation}
where $\zeta (s)$ is the Riemann zeta function and $L(s;\chi)$ is the Dirichlet
$L$-function with character $\chi$, and $F_q(T,X)\in\mathbb{Z}[X]$
is a polynomial with constant term 1.
The polynomial $F_q(T,X)$ is defined by the polynomial $g_{T}(X)$ in \cite[Theorem 13.6]{Sh-book}
for $K = F = \mathbb{Q}_{q}$, $\varepsilon'=1$ and $r=n$.

%%%%%%%%%
First we assume that $\boldsymbol{n}$ \textbf{is even}.
%%%%%%%%%Theorem2.2
\begin{Thm}
\label{even}
{\rm (Nagaoka \cite{Na})} Let $n$ be an even integer and $p$ be a prime number with
$p>n+3$ and $p \equiv (-1)^{n/2} \pmod{4}$. Then, for any odd integer $t\geq 1$, there
exists a modular form $F\in M_{\tfrac{n}{2}+\tfrac{p-1}{2}\cdot t}(\Gamma^{(n)})_{\mathbb{Z}_{(p)}}$
satisfying
$$
\varTheta (F) \equiv 0 \pmod{p}.
$$
Moreover $F$ is essential.
\end{Thm}
%%%%%%%%%%%%
\begin{Rem}
\label{exampleS}
(1)\, The modular form $F$  is realized by a constant multiple of Eisenstein series.
\\
(2)\,
In the case that $n=2$, $t=1$, and $p=23$, we obtain
$$
\varTheta (E_{12}^{(2)}) \equiv 0 \pmod{23}.
$$
(3)\;
There are several modular forms in $M_{12}(\Gamma^{(2)})$ satisfying a
congruence relation similar to that given in (2). For example,
$$
\varTheta (\vartheta_{\mathcal{L_{\text{Leech}}}}^{(2)}) \equiv 0 \pmod{23},
$$
where $\vartheta_{\mathcal{L_{\text{Leech}}}}^{(2)}$ is the degree 2 Siegel theta series attached
to the Leech lattice $\mathcal{L_{\text{Leech}}}$ (cf. \cite{N-T}). Moreover,
$$
\varTheta ([\Delta_{12}]) \equiv 0 \pmod{23},
$$
where $[\Delta_{12}]$ is the Klingen-Eisenstein series attached to the degree one
cusp form $\Delta_{12}\in S_{12}(\Gamma^{(1)})$ with $a(1;\Delta_{12})=1$ (cf. \cite{B}).
\\
(4) Let $\chi_{35}$ be the Igusa cusp form of degree 2 and weight $35$. It is known
that
$$
\varTheta (\chi_{35}) \equiv 0 \pmod{23},
$$
(cf. \cite{K-K-N}).
\end{Rem}
%%%%%%%%%%%
In the rest of this section, we treat the case that $\boldsymbol{n}$ \textbf{is odd}. We recall the
formula given in (\ref{explicit}). In this case, for $T\in\Lambda_n^+$, we have
\begin{equation*}
\begin{split}
% \label{explicitodd}
a(T;E_k^{(n)})& =A_{n,k}\cdot \prod_{\substack{q\mid D(T)\\q:\text{prime}}}F_q(T,q^{k-n-1}),\\
               & A_{n,k}:=2^{(n+1)/2}\cdot \zeta(1-k)^{-1}\prod_{i=1}^{(n-1)/2}\zeta(1+2i-2k)^{-1}.
\end{split}
\end{equation*}

Our first result is as follows.
%%%%%%%%%%%%%Theorem2.4%%%%%%%%%%%%%%%%%%%%%%%%%%%%%%%%%%%%%%%%%%%%%%%%%%%%%%%%%%%%%%%%%%%%%%%%%%%%%%%
\begin{Thm}
\label{main1}
Let $n$ be a positive integer such that $n \equiv 3 \pmod{8}$.
Assume that $p$ is a prime number such that $p>n$.
For any positive integer $t$,
We define a constant multiple of Siege-Eisenstein series $F_{k}^{(n)}$ by
\begin{equation*}
F_{k}^{(n)}:=p^{-\alpha_p(n,k)}\cdot E_k^{(n)}.
\end{equation*}
Here
\begin{align*}
  k&:=\frac{n+1}{2}+(p-1)\cdot t,\\
  \alpha_p(n,k)
   &:=
     \text{ord}_p(A_{n,k})=\text{ord}_p\left(\zeta(1-k)^{-1}\prod_{i=1}^{(n-1)/2}\zeta(1+2i-2k)^{-1}\right).
\end{align*}
Then for any positive integer $t$,
the modular form
$F_{k}^{(n)}$ has $\mathbb{Z}_{(p)}$ integral Fourier coefficients
and satisfies
$$
\varTheta (F_{k}^{(n)}) \equiv 0 \pmod{p}.
$$
Moreover, $F_{k}^{(n)}$ is essential.
\end{Thm}
\begin{Rem}
  By Theorem 3.5 in \cite{Bo-Ki-Ta},
  if $k = \frac{n + 1}{2} + (p - 1)$, then
  we have $\omega(F_{k}^{(n)}) = k$,
  where $\omega(F_{k}^{(n)})$ is the filtration of $F_{k}^{(n)} \mod{p}$.
\end{Rem}
%%%%%%%%%%%%%%%%%%
\begin{proof}
Using the theorem of von Staudt-Clausen
and the fact $p>n$, we see that all values $\zeta(1-k)$ and
$\zeta (1+2i-k)$\,$(1\leq i\leq\tfrac{n-1}{2})$ are $p$-integral. Therefore we have $\alpha_p(n,k)\leq 0$.
We prove that
$F_{k}^{(n)}$ satisfies the required properties:
\\
(i)\;\; $F_{k}^{(n)}$ has $p$-integral Fourier coefficients,\\
(ii)\;\; $\varTheta (F_{k}^{(n)}) \equiv 0 \pmod{p}$,\\
(iii)\;\;$F_{k}^{(n)}$ is essential, i.e., $a(T;F_{k}^{(n)})\not\equiv 0\pmod{p}$ for some $T\in\Lambda_n^+$.
\vspace{2mm}
\\
First we prove (i). The proof is reduced to show that $p^{-\alpha_p(n,k)}\cdot a(T;E_k^{(n)})$
is $p$-integral for any $T\in\Lambda_n$.

For $T\in\Lambda_n$ with
$\text{rank}(T)=r\leq n$, we have
$$
T[U]=\begin{pmatrix}T_1 & 0 \\ 0 & 0_{n-r}\end{pmatrix}\qquad
T_1\in\Lambda_r^+,\;\;\text{and}\;\;
U\in\text{GL}_n(\mathbb{Z}).
$$
We denote by $A_{r,k}(T)$ the zeta-$L$ factor
of $a(T;E_k^{(n)})$, i.e.,
\begin{equation*}
\begin{split}
\label{Ark}
A_{r,k}(T)=& \zeta(1-k)^{-1}\prod_{i=1}^{[\frac{r}{2}]}\zeta(1+2i-2k)^{-1}\\
                  & \times \begin{cases}
                     2^{r/2}\,L(1+\tfrac{r}{2}-k;\chi_{T_1})  &\text{($r$: even)}\\
                     2^{(r+1)/2}                             &\text{($r$: odd)}.
                     \end{cases}
\end{split}
\end{equation*}
When $r$ is odd, $p^{-\alpha_p(n,k)}\cdot A_{r,k}(T)=p^{-\alpha_p(n,k)}\cdot A_{r,k}$
is $p$-integral because $\text{ord}_p(A_{r,k}(T))=\text{ord}_p(A_{r,k})\geq \alpha_p(n,k)$.
Hence
$p^{-\alpha_p(n,k)}\cdot a(T;E_k^{(n)})$ is $p$-integral for $T\in\Lambda_n$ with
odd rank.

In the case that $r$ is even, the $L$-factor
$L(1+\tfrac{r}{2}-k;\chi_{T_1})$ appears in $A_{r,k}(T)$.
We prove that $L(1+\tfrac{r}{2}-k;\chi_{T_1})$ is $p$-integral for even $r$\,
$(2\leq r\leq n-1)$.

The following result is known regarding the $L$-value $L(1-m;\chi)$\,
($m\in\mathbb{N}$,\, $\chi$:\,quadratic).

For a prime number $p>2$, the value $L(1-m;\chi)$ is $p$-integral
except for the case that
the conductor of $\chi$ is equal to $p$ and $m$ is an
odd multiple of $(p-1)/2$. Moreover, if we exclude this exceptional case,
$L(1-m;\chi)$ is a rational integer (cf. \cite{Ca}, Theorem 3).

We shall show that the integer
$k-\tfrac{r}{2}$\;$(2\leq r\leq n-1,\,r:\,\text{even})$ cannot be an odd multiple of $(p-1)/2$.
If we assume that $k-\tfrac{r}{2}=\tfrac{n+1}{2}+(p-1)\cdot t-\tfrac{r}{2}$
is a multiple of $(p-1)/2$, then we have
$n+1-r$ is a multiple of $p-1$. By the assumption $p>n$, this is impossible. Therefore,
$L(1+\tfrac{r}{2}-k;\chi_{T_1})$ is a rational integer.
This implies that $p^{-\alpha_p(n,k)}\cdot A_{r,k}(T)$ is $p$-integral.
Consequently, we see that
$p^{-\alpha_p(n,k)}\cdot a(T;E_k^{(n)})$ is $p$-integral for any $T\in\Lambda_n$ with
even rank.

%%%
Secondly we prove (ii), namely,
$$
\varTheta (F_{k}^{(n)}) \equiv 0 \pmod{p}.
$$
To do this, it suffices to show that, if $T\in\Lambda_n^+$ satisfies
$\text{det}(T) \not\equiv 0 \pmod{p}$, then the corresponding Fourier
coefficient $a(T;F_{k}^{(n)})$ satisfies
\begin{equation}
\label{Theta}
a(T;F_{k}^{(n)}) \equiv 0 \pmod{p}.
\end{equation}
Our proof is based on Katsurada's functional equation for $F_q(T,X)$.
%%%%%%%%%%%%
\begin{Thm}
{\rm (Katsurada \cite{Ka})}
\label{Katsurada}
We assume that $n\in\mathbb{Z}_{>0}$ is odd, $q$ is a prime number, and
$T\in\Lambda_n^+$.
Then we have
\begin{equation}
\label{functionalK}
F_q(T,q^{-n-1}X^{-1})=\eta_q(T)(q^{(n+1)/2}X)^{-{\rm ord}_q(D(T))}F_q(T,X),
\end{equation}
where
$$
\eta_q(T)=h_q(T)({\rm det}(T),(-1)^{\tfrac{n-1}{2}}{\rm det}(T))_q(-1,-1)_q^{\tfrac{n^2-1}{8}},
$$
$h_q(T)$ is the Hasse invariant, and $(a,b)_q$ is the Hilbert symbol.
\end{Thm}
%%%%%%%%%%%%%
The following is a key lemma of our proof.
%%%%%%%%%%%%%
\begin{Lem}
\label{keylemmaS}
We assume that $n \equiv \pm 3 \pmod{8}$ and $T\in\Lambda_n^+$.
Then there is a prime divisor $q$ of $D(T)$ satisfying
$$
F_q(T,q^{-\tfrac{n+1}{2}})=0.
$$
\end{Lem}
%%%%%%%%%%%%
\noindent
{\it Proof of the lemma.}
By the assumption $n \equiv \pm 3 \pmod{8}$, we have
$$(-1,-1)_\infty^{\tfrac{n^2-1}{8}}=-1.
$$
This implies $\eta_\infty (T)=-1$.
By the product formula of Hilbert symbol (i.e., $\prod_{q\leq \infty}\eta_q(T)=1$),
we see that there is a prime $q$ such that $\eta_q(T)=-1$. For this $q$, we
substitute $q^{-\tfrac{n+1}{2}}$ for $X$ in (\ref{functionalK}). This shows
$F_q(T,q^{-\tfrac{n+1}{2}})=0$, which completes the proof of the lemma.
\begin{flushright}
$\square$
\end{flushright}
%%%%%%%%%%%%%%
% \newpage
%%Table
\begin{Ex}
We give a short table of $\prod F_q(T,X)$ in the case that $n=3$.
%These shows the validity of Lemma \ref{keylemmaS}.
\begin{table}[hbtp]
\caption{Example of $\prod F_q(T,X)$ in the case $n=3$ and $D(T)\leq 12$}
\begin{center}
\begin{tabular}{c|c||c|c}
$D(T)$ & $\prod F_q(T,X)$ & $D(T)$ & $\prod F_q(T,X)$ \\ \hline
$2$       & $ 1-2^2X$             & $9$         & $1-3^4X^2$ \\
$3$       & $1-3^2X$              & $10_1$    & $(1-2^2X)(1+5^2X)$ \\
$4$       & $1-2^4X^2$           & $10_2$    & $(1+2^2X)(1-5^2X)$ \\
$5$       & $1-5^2X$              & $11$       & $1-11^2X$             \\
$6_1$   & $(1+2^2X)(1-3^2X)$  & $12_1$   & $(1-2^2X+2^4X^2)(1-3^2X)$ \\
$6_2$   & $(1-2^2X)(1+3^2X)$  & $12_2$   &  $(1+2^2X+2^4X^2)(1-3^2X)$ \\
$7$      & $1-7^2X$               & $12_3$    & $(1+2^4X^2)(1-3^2X)$ \\
$8_1$   & $(1-2^2X)(1+2^4X^2)$ & $12_4$   & $(1-2^4X^2)(1+3^2X)$ \\
$8_2$   & $1-2^6X^3$               & $13$          & $1-13^2X$
\end{tabular}
\end{center}
\end{table}
\\
Here we used a suffix notation $D(T)_i$ when the $T$ has multiple genera.
The index is distinguished by their $2$-adic types (cf. \cite{K-W}).
\end{Ex}
%%%%%%%%%%
\quad
We return to the proof of the theorem \ref{main1}.
We assume that
$n \equiv 3 \pmod{8}$
and prove
$a(T;F_{k}^{(n)}) \equiv 0 \pmod{p}$ under the condition
$\text{det}(T)\not\equiv 0\pmod{p}$.
(We need to exclude the case $n \equiv -3 \pmod{8}$
because the weight $k$ must be even.)
The condition $\text{det}(T)\not\equiv 0\pmod{p}$ implies that
$D(T)\not\equiv 0\pmod{p}$. Hence, by Lemma \ref{keylemmaS}, we have
\begin{align*}
\prod_{q\mid D(T)}F_q(T,q^{k-n-1})
              &=\prod_{q\mid D(T)}F_q(T,q^{-\tfrac{n+1}{2}+(p-1)\cdot t})\\
                                             &\equiv \prod_{q\mid D(T)}F_q(T,q^{-\tfrac{n+1}{2}})
                                               =0 \pmod{p}.
\end{align*}
Since $p^{-\alpha_p(n,k)}\cdot A_{n,k}(T)$ is a $p$-integral (in particular $p$-adic unit),
we obtain
$$
a(T;F_{k}^{(n)}) \equiv 0 \pmod{p}.
$$
Finally we shall prove that $F_{k}^{(n)}$ is essential.
\\
For this purpose, it suffices to show that there is a matrix $T\in\Lambda_n^{+}$
such that $D(T)=p$ because we can prove
$a(T;F_{k}^{(n)}) \not\equiv 0 \pmod{p}$ for such $T$ (note that $p^{-\alpha_p(n,k)}\cdot A_{n,k}(T)$
is $p$-adic unit).

We set $n=8s+3$.
In the ternary case, it is known that there is a matrix $T_1\in\Lambda_3^+$
satisfying $D(T_1)=2^2\text{det}(T_1)=p$ for any prime number $p$
(cf. Remark \ref{remarkS}, (2)). We set
$$
T=T_1\perp \underbrace{\tfrac{1}{2}U\perp\cdots \perp \tfrac{1}{2}U}_{s\, \text{times}}
\in \Lambda_n^+,
$$
where
$U$ is a
positive-definite even unimodular symmetric matrix of rank $8$.
Then the matrix $T$ satisfies the required property
$D(T)=p$.
This shows that $F_{k}^{(n)}$ is essential and
completes the proof of Theorem \ref{main1}.
%\begin{flushright}
%$\square$
%\end{flushright}
\end{proof}
%%%%%%%%%%%%%%%Remark2.8
\begin{Rem}
\label{remarkS}
(1)\;
We consider the Eisenstein series
$$
E_k^{(n)}(Z,s)=\sum_{\binom{*\;*}{C\,D}\in \Gamma_\infty^{(n)}\backslash \Gamma^{(n)}}
\text{det}(CZ+D)^{-k}|\text{det}(CZ+D)|^{-s},\quad (Z,s)\in\mathbb{H}_n\times\mathbb{C}.
$$
The analytic properties of this series were studied by Weissauer, Shimura, and others.
Weissauer proved the following (cf. \cite{Weis}, $\S$ 14):

If $\frac{n+1}{2} \equiv 2 \pmod{4}$, then
$E_{\tfrac{n+1}{2}}^{(n)}(Z,s)$ is holomorphic at $s=0$; moreover,
\begin{equation}
\label{SEisenvanish}
E_{\tfrac{n+1}{2}}^{(n)}(Z,0) \equiv 0\quad(\text{identically vanishes}).
\end{equation}
Since the condition $\frac{n+1}{2} \equiv 2 \pmod{4}$ is equivalent to
$n \equiv 3 \pmod{8}$,  our Lemma \ref{keylemmaS} shows that
$E_{\tfrac{n+1}{2}}^{(n)}(Z,0)$ is a singular modular form, and thus
it identically vanishes (note that $(n+1)/2$ is not a singular weight).
Namely, Lemma \ref{keylemmaS} gives another
proof of (\ref{SEisenvanish}).\\
(2)\; In the proof of Theorem \ref{main1}, we used the fact that
there is an element $T_1\in\Lambda_3^+$ such that $D(T_1)=p$ for
any prime number $p$. In fact, we may take $T_1$ as follows:
$$
T_1=
\begin{pmatrix}
1 & \frac{1}{2} & \frac{1}{2}\\
\frac{1}{2} & 1 &  0  \\
\frac{1}{2} & 0 & 1
\end{pmatrix}\;\text{for}\;p=2,\;\;
T_1=
\begin{pmatrix}
1 & 0 & 0\\
0 & 1 &  \frac{1}{2}  \\
0 & \frac{1}{2} & \frac{p+1}{4}
\end{pmatrix}
\;\text{for}\;p\;\text{with}\;
p \equiv -1 \pmod{4}.
$$
In the case $p \equiv 5 \pmod{8}$, we may set
$T_1=
\begin{pmatrix}
1 & 0 & \frac{1}{2} \\
0 & 2 &  \frac{1}{2}  \\
\frac{1}{2} & \frac{1}{2} & \frac{p+3}{8}.
\end{pmatrix}$. Finally
we consider the case $p \equiv 1 \pmod{8}$. The following is
due to Schulze-Pillot:\\
Choose a prime $q$ with $q \equiv 3 \pmod{4}$,
$\left(\frac{p}{q}\right)=\left(\frac{q}{p}\right)=-1$, and
$a\in\mathbb{Z}$ with $a^2 \equiv -p \pmod{q}$.
Set
$$
T_1=
\begin{pmatrix}
\frac{a^2q+a^2+p}{q} & -a         & \frac{-a(q+1)}{2} \\
           -a            &  1          & \frac{q}{2}          \\
 \frac{-a(q+1)}{2}   & \frac{q}{2} & \frac{q(q+1)}{4}
\end{pmatrix}.
$$
Then  $T_1$ is positive definite and $D(T_1)=p$.
\end{Rem}
%%%%%%%%%%%%%%%%%%%%%%%%%%%%%%%%%%%%%%%%%%%%%%%%%%%%%%%%%%%%%%%%%%%%%%%%%
%%%%%%%%%%%%%%%%%%%%%%%%%%%%%%%%%%%%%%%%%%%%%%%%%%%%%%%%%%%%%%%%%%%%%%%%%
\section{Hermitian modular case}
Let $m$ be a positive integer and $\boldsymbol{K}=\mathbb{Q}(\sqrt{-D_{\boldsymbol{K}}})$
an imaginary quadratic field with discriminant $-D_{\boldsymbol{K}}<0$.
We denote by $\mathcal{O}_{\boldsymbol{K}}$ the ring of integers of $\boldsymbol{K}$.
Let $\chi_{\boldsymbol{K}}$ be the quadratic Dirichlet character of conductor $D_{\boldsymbol{K}}$
corresponding to the extension $\boldsymbol{K}/\mathbb{Q}$ by the global class field
theory. Denote by $\underline{\chi}_{\boldsymbol{K}}=\prod_v \underline{\chi}_{\boldsymbol{K},v}$
the idele class character which corresponds to $\chi_{\boldsymbol{K}}$.
%%%%%%%%%
\subsection{Hermitian modular forms}
For a $\mathbb{Q}$-algebra $R$, the group $SU(m, m)(R)$ is given
as
$$
SU(m, m)(R)=\left\{ g\in \text{SL}_{2m}(R\otimes_{\mathbb{Q}}\boldsymbol{K})\,
             \Big{|}\,
             g^*\begin{pmatrix} 0_m & -1_m\\ 1_m & 0_m\end{pmatrix}
             g
             =\begin{pmatrix} 0_m & -1_m\\ 1_m & 0_m\end{pmatrix}
             \right\},
$$
where $g^*={}^t\overline{g}$.

We set
$$
\Gamma_{\boldsymbol{K}}^{(m)}
= SU(m, m)(\mathbb{Q})\cap\text{SL}_{2m}  (\mathcal{O}_{\boldsymbol{K}}).
$$
We denote by $M_k(\Gamma_{\boldsymbol{K}}^{(m)})$ the space of Hermitian
modular forms of weight $k$ for $\Gamma_{\boldsymbol{K}}^{(m)}$. Any modular
form $F$ in $M_k(\Gamma_{\boldsymbol{K}}^{(m)})$ has a Fourier expansion
of the form
$$
F(Z)=\sum_{0\leq H\in\Lambda_m(\mathcal{O}_{\boldsymbol{K}})}a(H;F)q^H,\quad
q^H=\text{exp}(2\pi i\text{tr}(HZ)),\quad Z\in\mathcal{H}_m,
$$
where
\begin{align*}
& \mathcal{H}_m=\{\,Z\in M_m(\mathbb{C})\,\mid\, \tfrac{1}{2i}(Z-Z^{*})>0\,\}
\;(\text{the Hermitian upper half space}),\\
  & \Lambda_m(\mathcal{O}_{\boldsymbol{K}})=
    \{\,H=(h_{jl})\in M_m(\boldsymbol{K})\,\mid\,
    H^{*} = H, \
    h_{jj}\in\mathbb{Z},\,
    \sqrt{-D_{\boldsymbol{K}}}\,h_{jl}\in\mathcal{O}_{\boldsymbol{K}}\,\}.
\end{align*}
We also set $\Lambda_m^+(\mathcal{O}_{\boldsymbol{K}})=\{\,H\in
\Lambda_m(\mathcal{O}_{\boldsymbol{K}})\,\mid\, H>0\,\}$.
\\
\\
We can also define the theta operator as in the case of Siegel modular forms:
$$
\varTheta : F=\sum a(H;F)q^H\,\longmapsto
\varTheta (F):=\sum a(H;F)\cdot\text{det}(H)q^H.
$$
%%%%%%%%%%%%%%%%%%%%%%%%%%%%%%%%%%%%%%%%%%%%%%%%%
\subsection{Hermitian Eisenstein series}
We set
$$
\Gamma_{\boldsymbol{K},\infty}^{(m)}=
\left\{\begin{pmatrix}A&B\\ C&D\end{pmatrix}\in \Gamma_{\boldsymbol{K}}^{(m)}\,\Big{|}\,
C=0_m\,\right\}.
$$
For a positive even integer $k > 2m$, we define Eisenstein series of weight $k$ by
$$
\mathcal{E}_{k}^{(m)}(Z)
=\sum_{M=\binom{*\;*}{C\,D}\in\Gamma_{\boldsymbol{K},\infty}^{(m)}\backslash
\Gamma_{\boldsymbol{K}}^{(m)}}
\text{det}(CZ+D)^{-k},\quad Z\in \mathcal{H}_m.
$$
For a prime number $q$, we set
$\mathcal{O}_{\boldsymbol{K},q} = \mathcal{O}_{\boldsymbol{K}}\otimes_{\mathbb{Z}}\mathbb{Z}_{q}$
and set
\begin{align*}
  & \Lambda_{m}(\mathcal{O}_{\boldsymbol{K}, q})\\
  & = \left\{
    H = (h_{jl}) \in M_{m}(\boldsymbol{K}\otimes_{\mathbb{Q}}\mathbb{Q}_{q})
    \bigm |
    H^{*} = H, \
    h_{jj} \in \mathbb{Z}_{q}, \
    \sqrt{-D_{\boldsymbol{K}}}\,h_{jl} \in
    \mathcal{O}_{\boldsymbol{K}, q}
  \right\}.
\end{align*}
Let $H \in \Lambda_{m}(\mathcal{O}_{\boldsymbol{K}})$ with $H \geq 0$ and set
$r = \text{rank}_{\boldsymbol{K}} H$.
For each prime number $q$, take $U_{q} \in \text{GL}_{m}(\mathcal{O}_{\boldsymbol{K},q})$ so that
\begin{equation}
  \label{eq:huq-prime}
  H[U_{q}] = \begin{pmatrix}H_{q}'& 0 \\ 0 & 0 \end{pmatrix}
\end{equation}
with $H_{q}' \in \Lambda_{r}(\mathcal{O}_{\boldsymbol{K},q})$.
Here, for $A, B \in \mathrm{Res}_{\boldsymbol{K}/\mathbb{Q}}M_{n}$, we define
\begin{equation*}
  A[B] := B^{*}AB.
\end{equation*}
For $H \in \Lambda_{r}(\mathcal{O}_{\boldsymbol{K},q})$ with $\det H \ne 0$,
we denote by $\mathcal{F}_q(H, X) \in \mathbb{Z}[X]$
the polynomial given in \cite{Ikeda}, $\S$  2 (Ikeda denotes it by $F_{p}(H; X)$).
Then the polynomial $\mathcal{F}_q(H_{q}', X)$ does not depend
on the choice of $U_{q}$. Therefore, we denote it by $\mathcal{F}_q(H, X)$.
For $H \in \Lambda_{r}(\mathcal{O}_{\boldsymbol{K}})$ (resp. $\in
\Lambda_{r}(\mathcal{O}_{\boldsymbol{K}, q})$) with $\det H \ne 0$, we define
\begin{equation*}
  \gamma(H)=(-D_{\boldsymbol{K}})^{[r/2]}\text{det}(H)\in
  \mathbb{Z} \quad
  (\text{resp.} \in \mathbb{Z}_{q}).
\end{equation*}

%%%%%%%%%%%%%%%%%%%%%%%%%%%%
\begin{Thm}
  \label{thm:explicitH}
  Let $H \in \Lambda_{m}(\mathcal{O}_{\boldsymbol{K}})$ with $H \ge 0$ and set
  $r = {\rm rank}_{\boldsymbol{K}} H$.
  Then the $H$th Fourier coefficient $a(H; \mathcal{E}_{k}^{(m)})$ of
  the Hermitian Eisenstein series $\mathcal{E}_{k}^{(m)}$ is given as follows:
  \begin{equation}
  \label{explicitH}
    2^{r}
    \left( \prod_{i=1}^{r}L(i-k, \chi_{\boldsymbol{K}}^{i-1})^{-1} \right)
    \left( \prod_{q: {\rm prime}}
      \mathcal{F}_q(H; q^{k - 2r}) \right).
  \end{equation}
  Here we understand that $L(i-k, \chi_{\boldsymbol{K}}^{i-1}) = \zeta(i-k)$ if $i$ is odd.
\end{Thm}
\begin{Rem}
(1)\;The product over all primes $q$ is actually a finite product.
  The polynomial $\mathcal{F}_q(H; X)$ is a $\mathbb{Z}$-coefficient polynomial of
  degree $\mathrm{ord}_{q}\left(\gamma(H_{q}'))\right)$ with the constant term $1$
  (see Theorem \ref{Ikeda}). Here $H_{q}'$ is the matrix in \eqref{eq:huq-prime}.\\
(2)\;This formula is also stated in \cite{Ikeda} for the case $\det H \ne 0$.
\end{Rem}
We shall prove Theorem \ref{thm:explicitH} in \S \ref{sec:proof-theor-refthm:1}.
%%%%%%%%%%%%%
%where $\gamma(H)=(-D_{\boldsymbol{K}})^{[m/2]}\text{det}(H)\in\mathbb{Z}$ and
%$\mathcal{F}_q(H,X)\in\mathbb{Z}[X]$ is a polynomial with the constant term 1.It should be
%noted that $\gamma(H)<0$ if $m \equiv 2 \pmod{4}$.
The second main result is as follows.
%%%%%%%%%%%%%%
\begin{Thm}
\label{main2}
Let $m$ be a positive integer such that $m \equiv 2 \pmod{4}$.
Assume that $p>m+1$ is a prime number such that $D_{\boldsymbol{K}} \not\equiv 0 \pmod{p}$.
For any positive integer $t$,
We define a constant multiple of Eisenstein series by
$$
G_{k}^{(m)}:=p^{-\beta_p(m,k)}\cdot \mathcal{E}_{k}^{(m)}.
$$
Here
\begin{align*}
  k&:=m+(p-1)\cdot t\\
  \beta_p(m,k)&:=\text{ord}_p\left(\prod_{i=1}^m L(i-k,\chi_{\boldsymbol{K}}^{i-1})^{-1} \right),
\end{align*}
where $L(i-k, \chi_{\boldsymbol{K}}^{i-1}) = \zeta(i-k)$ if $i$ is odd as in Theorem \ref{explicitH}.
Then for any positive integer $t$,
the modular form $G_{k}^{(m)}$ has $\mathbb{Z}_{(p)}$ integral Fourier coefficients and
satisfies
$$
\varTheta (G_{k}^{(m)}) \equiv 0 \pmod{p}.
$$
Moreover $G_{k}^{(m)}$ is essential.
\end{Thm}
%%%%%%%%%%%%%%%%%%%%
\begin{proof}
Since $p>m+1$, we see that the weight $k=m+(p-1)\cdot t$ is
greater than $2m$, so the condition on the convergence of
$\mathcal{E}_{k}^{(m)}$ is fulfilled.

By the assumptions on $p$ and $k$, we see that the each factor $L(i-k, \chi_{\boldsymbol{K}}^{i-1})$
is $p$-integral, so $\beta_p(m,k)\leq 0$.

First we prove the $p$-integrality of $G_{k}^{(m)}$. We set
$$
B_{r,k}:=\prod_{i=1}^{r}L(i-k, \chi_{\boldsymbol{K}}^{i-1})^{-1},
$$
which is the $L$-facor appearing in the Fourier coefficient $a(H;\mathcal{E}_{k}^{(m)})$
for $H$ for $r=\text{rank}(H)$. Since each factor $L(i-k,\chi_{\boldsymbol{K}}^{i-1})$ is
$p$-integal, we see that $p^{-\beta_p(m,k)}\cdot B_{r,k}$ is $p$-integral. Consequently,
$p^{-\beta_p(m,k)}\cdot a(H;\mathcal{E}_{k}^{(m)})$ is $p$-integral for any
$H\in \Lambda_m(\mathcal{O}_{\boldsymbol{K}})$.

Next we show that
$
\varTheta(G_{k}^{(m)})
\equiv 0 \pmod{p}.
$
As in the case of Siegel modular forms, it is sufficient to show  the following:

If $H\in\Lambda_m^+(\mathcal{O}_{\boldsymbol{K}})$ satisfies
$\text{det}(H) \not\equiv 0 \pmod{p}$, then
\begin{equation}
\label{congH}
a(H;G_{k}^{(m)}) \equiv 0 \pmod{p}.
\end{equation}
For the proof, we use the functional equation for $\mathcal{F}_q(H,X)$ due to Ikeda.
%%%%%%%%%%%%
\begin{Thm}{\rm (Ikeda \cite{Ikeda})}
  \label{Ikeda}
  For $H\in\Lambda_m(\mathcal{O}_{\boldsymbol{K, q}})$ with
  ${\rm det}(H)\ne 0$, the polynomial $\mathcal{F}_q(H,X)$ has the functional equation
  \begin{equation}
    \label{functionalH}
    \mathcal{F}_q(H,q^{-2m}X^{-1})=\underline{\chi}_{\boldsymbol{K}, q}(\gamma(H))^{m-1}(q^mX)^{-{\rm ord}_q(\gamma(H))}
    \mathcal{F}_q(H,X).
  \end{equation}
\end{Thm}
The following is a key lemma in the case of Hermitian modular forms.
\begin{Lem}
\label{keylemmaH}
Assume that $m \equiv 2 \pmod{4}$ and $H\in\Lambda_m^+(\mathcal{O}_{\boldsymbol{K}})$.
Then there is a prime divisor $q$ of $\gamma(H)$ such that
\begin{equation}
\label{vanishingH}
\mathcal{F}_q(H,q^{-m})=0.
\end{equation}
\end{Lem}
\noindent
{\it Proof of the lemma.}
By $m \equiv 2 \pmod{4}$, we see that $\gamma(H)<0$, and so
$\underline{\chi}_{\boldsymbol{K}, \infty}(\gamma(H))=-1$. By the product formula of the idele
class character, there is a prime number $q$ such that
$\underline{\chi}_{\boldsymbol{K}, q}(\gamma(H))=-1$. In view of the functional equation
(\ref{functionalH}), we obtain $\mathcal{F}_q(H,q^{-m})=0$.
\begin{flushright}
$\square$
\end{flushright}
%%%%%%%%%%
%%%%%%%%%%%%%%%
We return to the proof of (\ref{congH}).
Since $\text{det}(H) \not\equiv 0 \pmod{p}$,
we see that $\gamma(H)\not\equiv 0 \pmod{p}$.
(It should be noted that $D_{\boldsymbol{K}} \not\equiv 0 \pmod{p}$.)
This implies that
\begin{align*}
\prod_{q\mid \gamma(H)}\mathcal{F}_q(H,q^{k-2m})&=
\prod_{q\mid \gamma(H)}\mathcal{F}_q(H,q^{-m+(p-1)\cdot t})\\
& \equiv \prod_{q\mid \gamma(H)}\mathcal{F}_q(H,q^{-m})=0 \pmod{p}.
\end{align*}
Therefore,
\begin{align*}
a(H;G_{k}^{(m)}) & =(\text{a}\, p\text{-adic integer}\,C)\times \prod_{q\mid \gamma(H)}\mathcal{F}_q(H,q^{k-2m})\\
       & \equiv C\times 0=0 \pmod{p}.
\end{align*}
Finally we prove that $G_{k}^{(m)}$ is essential.
It is enough to show the existence of
$H\in\Lambda_m^+(\mathcal{O}_{\boldsymbol{K}})$ with $\gamma(H)=-p$ because
we have $a(H;G_{k}^{(m)}) \not\equiv 0 \pmod{p}$ for such $H$. Namely we can prove that
$p^{-\beta_p(m,k)}\cdot B_{m,k}$ is $p$-adic unit and
$$\prod_{q\mid \gamma(H)}\mathcal{F}_q(H,q^{k-2m}) \not\equiv 0 \pmod{p}
$$ for such $H$.

We set $m=4s+2$.
First we take a matrix $H_1\in\Lambda_2^+(\mathcal{O}_{\boldsymbol{K}})$ with
$\gamma(H_1)=-p$
(for the existence of $H_1$, see, e.g., \cite{M-N}, Lemma 3.1).

Next we take a positive-definite even unimodular Hermitian matrix $W$ of
rank 4 such that $\text{det}(W)=(2/\sqrt{D_{\boldsymbol{K}}})^4$.
An explicit formula of such a matrix is given in
\cite{D-K}, Lemma 1. Then the matrix
$$
H=H_1\perp \underbrace{\tfrac{1}{2}W\perp\cdots \perp \tfrac{1}{2}W}_{s\,\text{times}}
\in \Lambda_m^+(\mathcal{O}_{\boldsymbol{K}}),
$$
satisfies $\gamma(H)=-p$.
This shows that $G_{k}^{(m)}$ is essential and
completes the proof.
\end{proof}
%%%%%%%%%%%%%%%%
\begin{Rem}
(1) As in the case of Siegel modular forms, we consider the Eisenstein series
\begin{align*}
&
\mathcal{E}_{k}^{(m)}(Z,s)
=\sum_{M=\binom{*\;*}{C\,D}\in\Gamma_{\boldsymbol{K},\infty}^{(m)}\backslash
\Gamma_{\boldsymbol{K}}^{(m)}}
\text{det}(CZ+D)^{-k}|\text{det}(CZ+D)|^{-s},\\
& (Z,s)\in \mathcal{H}_m\times\mathbb{C}.
\end{align*}
We assume that $m \equiv 2 \pmod{4}$. In this case,
it is known that $\mathcal{E}_{m}^{(m)}(Z,s)$ is holomorphic in $s$ (e.g., cf. Shimura \cite{Sh}).
Moreover, by Lemma \ref{keylemmaH}, we have
$$
\mathcal{E}_m^{(m)}(Z,0) \equiv 0\quad(\text{identically vanishes}).
$$
(2) For the case that $m=2$, the mod $p$ vanishing property of
$\varTheta (\mathcal{E}_k^{(2)})$ has previously been studied (Kikuta-Nagaoka \cite{K-N}).
\end{Rem}
%%%%%%%%
%%%%%%%%%%%%%%%%%%%%%%%%%%%%%%%%%%%%%%%%%%%%%%%%%%%%%%
\subsection{Proof of Theorem \ref{thm:explicitH}}
  In this proof, we denote $SU(m, m)$ by $G_{m}$.
  For $g \in G_{m}$, we define $a_{g}, b_{g}, c_{g}, d_{g} \in M_{m}$, such that
  \begin{math}
    g =
    \begin{pmatrix}
      a_{g} & b_{g}\\
      c_{g} & d_{g}
    \end{pmatrix}.
  \end{math}
  For each place $v$ of $\mathbb{Q}$, we set $\boldsymbol{K}_{v}
  = \boldsymbol{K} \otimes_{\mathbb{Q}}\mathbb{Q}_{v}$.
  For a $\mathbb{Q}$-algebra $R$, we set
  \begin{equation*}
    S_{m}(R) =
    \left\{ g \in M_{m}(R \otimes_{\mathbb{Q}}\boldsymbol{K})\;\Big{|}\;
      g^{*} = g
    \right\}.
  \end{equation*}
  For $x \in S_{m}$, we set
  \begin{equation*}
    \nu_{m}(x) =
    \begin{pmatrix} 1_m & x \\ 0_m & 1_m\end{pmatrix}
    \in G_{m}.
  \end{equation*}
  For $\alpha \in \text{Res}_{\boldsymbol{K}/\mathbb{Q}}\text{GL}_{m}$ with
  $\det \alpha = \det \overline{\alpha}$, we set
  \begin{equation*}
    \mu_{m}(\alpha) =
    \begin{pmatrix} \alpha & 0_m \\ 0_m & \left( \alpha^{*} \right)^{-1}
    \end{pmatrix}
    \in G_{m}.
  \end{equation*}
  We define the Siegel parabolic subgroup $P_{m}$ of $G_{m}$ as follows.
  \begin{equation*}
    P_{m} = \left\{
      g \in G_{m}\Big{|} c_{g} = 0_m
    \right\}.
  \end{equation*}
  For a place $v$ of $\mathbb{Q}$, we define a maximal compact subgroup $C_{v}$ as follows:
  \begin{equation*}
    C_{v} =
    \begin{cases}
      \left\{g \in G_{m}(\mathbb{R})\bigm | g \cdot i = g \right\} & \text{if } v = \infty,\\
      G(\mathbb{Q}_{v}) \cap \text{GL}_{2m}(\mathcal{O}_{\boldsymbol{K},q}) & \text{if } v < \infty.
      \ \end{cases}
  \end{equation*}
  Then the Iwasawa decomposition $G_{m}(\mathbb{Q}_{v}) = P_{m}(\mathbb{Q}_{v}) C_{v}$
  for each place $v$ of $\mathbb{Q}$ holds.
  For each place $v$ of $\mathbb{Q}$, we define a function $\phi_{n, v}$ on
  $G_{m}(\mathbb{Q}_{v})$ as follows.
  \begin{equation*}
    \phi_{n, v}(yw) = \left|\text{det}\, a_{y}\right|_{v}^{k},
  \end{equation*}
  where $y \in P_{m}(\mathbb{Q}_{v})$ and $w \in C_{v}$.
  We note that $\det a_{y} \in \mathbb{Q}_{v}^{\times}$ by \cite{Sh}, Lemma 1.1.
  Here we take the norm $|\cdot|_{v}$ so that $\left|\cdot\right|_{\infty}$
  is the usual Euclidean norm of $\mathbb{R}$
  and $\left|q\right|_{v}=q^{-1}$ if $v = q$ is a finite place.
  For each place $v$ of $\mathbb{Q}$,
  we take a Haar measure $\mu_{v}(x)$ on $S_{m}(\mathbb{Q}_{v})$ as in \cite{Sh}
  $\S$ 3, (3. 19).
  Further, for each place $v$ of $\mathbb{Q}$, we take an additive character $\textbf{e}_{v}$
  of $\mathbb{Q}_{v}$ by
  $\textbf{e}_{\infty}(x) = \textbf{e}(x)$ for $x \in \mathbb{R}$.
  If $v = q$ is a finite place, then we take $\textbf{e}_{q}$ so that
  $\textbf{e}_{q}(x) = \textbf{e}(-x)$ for $x \in \mathbb{Z}[q^{-1}]$.

  Let $H \in \Lambda_{m}(\mathcal{O}_{\boldsymbol{K}})$ with $H \geq 0$ and set $r = \text{rank}\, H$.
  We take $U \in \text{SL}_{m}(\boldsymbol{K})$ such that
  \begin{equation*}
    H[U] = \begin{pmatrix} H' & 0 \\ 0 & 0_{m-r}\end{pmatrix}.
  \end{equation*}
  For each place $v$ of $\mathbb{Q}$,
  we take matrices $U_{v}$ and $H_{v}'$ as follows.
  For each prime number $q$, we take $U_{q} \in \text{SL}_{m}(\mathcal{O}_{\boldsymbol{K},q})$ so that
  \begin{equation*}
    H[U_{q}] = \begin{pmatrix} H'_{q} & 0 \\ 0 & 0_{m-r}\end{pmatrix}.
  \end{equation*}
  We take $U_{\infty} \in SU_{m}(\mathbb{C})$ so that
  \begin{equation*}
    H[U_{\infty}] = \begin{pmatrix} H'_{\infty} & 0 \\ 0 & 0_{m-r}\end{pmatrix}.
  \end{equation*}
  Then for each place $v$ of $\mathbb{Q}$, by the choice of $U_{v}$, there exists
  $\alpha_{v} \in \text{GL}_{r}(\boldsymbol{K}_{v})$ and
  $\beta_v \in \text{GL}_{m-r}(\boldsymbol{K}_{v})$ such that
  \begin{equation*}
    U^{-1}U_{v} = \begin{pmatrix} \alpha_v & 0 \\ * & \beta_v\end{pmatrix}.
  \end{equation*}

  Then by a similar argument to that in \cite{Sh}, \cite[Proposition 4.2]{Take},
  (though there is a missing factor in \cite[Proposition 4.2]{Take}),
  we have the following.
  \begin{equation*}
    a(H; \mathcal{E}_k^{(m)})\textbf{e}(i\text{tr} HY) =
    c_{\mu}
    a_{\infty}(H, Y, k)
    \prod_{q: \text{ prime}}a_{q}(H, k).
  \end{equation*}
  Here we set $Y = \tfrac{Z - Z^{*}}{2i}$.
  The factor $c_{\mu}$ is given as
  \begin{equation*}
    c_{\mu} = 2^{r(r-1)/2}D_{\boldsymbol{K}}^{-r(r-1)/4}.
  \end{equation*}
  The factor $a_{\infty}(H, Y, k)$ is given as follows.
  \begin{align*}
    & \left( \text{det} Y \right)^{-k/2}\times \\
    &\int_{S_{r}(\mathbb{R})}
      \phi_{\infty}\left(
      w_{m, r}
      \nu_{m}\text{diag}(x, 0_{m-r})
      \mu_{m}(U^{-1} Y^{1/2})
      \right)
      \textbf{e}_{\infty}(-\text{tr} H' x)d\mu_{\infty}(x).
  \end{align*}
  The factor $ a_{q}(H, k)$ is given as follows.
  \begin{equation*}
    \int_{S_{r}(\mathbb{Q}_{q})}
    \phi_{q}\left(w_{m, r}
      \nu_{m}(\text{diag}(x, 0_{m-r}))
      \mu_{m}(U)^{-1}
    \right)
    \textbf{e}_{q}\left(-\text{tr} H'x\right)d\mu_{q}(x).
  \end{equation*}
  Here $w_{m, r}$ is given by
  \begin{equation*}
    \begin{pmatrix}
      0_{r} & & -1_{r} & \\
      & 1_{m-r} & & 0_{m-r}\\
      1_{r} & & 0_{r} & \\
      & 0_{m-r} & & 1_{m-r}
    \end{pmatrix}.
  \end{equation*}
  For a place $v$ of $\mathbb{Q}$, we have
  $\mathrm{tr}(xy) \in \mathbb{Q}_{v}$ for $x, y \in S_{r}(\mathbb{Q}_{v})$.
  We note that we can consider $\textbf{e}_{v}(-\text{tr} H' x)$ for $x \in S_{r}(\mathbb{Q}_{v})$.
  Let $v=q$ be a finite place.
  By replacing $x$ by $x[\alpha_{v}^{*}]$ and noting that there exists $\gamma \in P_{m}(\mathbb{Q}_{v})$
  such that $\text{det}\, a_{\gamma} = (\text{det}\, \overline{\alpha}_{v}\alpha_{v})^{-1}$ and
  \begin{equation*}
    w_{m, r}\nu_{m}\left(\text{diag}\left(x[\alpha_{v}^{*}], 0\right)\right)\mu_{m}(U^{-1})
    = \gamma w_{m, r}\nu_{m}\left(\text{diag}(x, 0)\right)\mu_{m}(U_{v}^{-1}),
  \end{equation*}
  we have
  \begin{align*}
    a_{q}(H, k) &=
                  \left|\text{det}(\alpha_{v}\alpha_{v}^{*})\right|_{v}^{r-k}
                  \int_{S_{r}(\mathbb{Q}_{v})}\phi_{q}\left(w_{m, r}\nu_{m}\left(\text{diag}(x, 0)\right)\right)
                  \textbf{e}_{q}(-\text{tr}H_{q}'x)d\mu_{q}(x)\\
                &=\left|\det(\alpha_{v}\alpha_{v}^{*})\right|_{v}^{r-k}
                  \int_{S_{r}(\mathbb{Q}_{v})}\phi_{q}\left(w_{r}\nu_{r}(x)\right)
                  \textbf{e}_{q}(-\text{tr}H_{q}'x)d\mu_{q}(x),
  \end{align*}
  where $w_{r} = w_{r, r}\in G_{r}$.
  As is well known, this can be written as follows (cf. \cite{Sh}, \cite{Ikeda}):
  \begin{equation*}
    a_{q}(H, k) =
    \left|\det(\alpha_{v}\alpha_{v}^{*})\right|_{v}^{r-k}
    L_{r, q}(k) \mathcal{F}_q(H_{q}', q^{-k}),
  \end{equation*}
  where
  \begin{equation*}
    L_{r, q}(k) = \prod_{i=0}^{r-1}(1 - \chi_{\boldsymbol{K}}^{i}(q)q^{i-k} ).
  \end{equation*}
  Here we understand $\chi_{\boldsymbol{K}}^{i}(q) = 1$ if $i$ is even.
  By a similar computation at the infinite place, we have the following.
  \begin{equation*}
    a_{\infty}(H, Y, k) =
    \left|\text{det}(\alpha_{v} \alpha_{v}^{*})\right|_{v}^{r-k}
    \int_{S_{r}(\mathbb{R})}\left|\det (x + \eta i)\right|_{\infty}^{-k}
    \textbf{e}_{\infty}(-\text{tr}H_{\infty}'x)
    d\mu_{\infty}(x).
  \end{equation*}
  Here $\eta$ is the $r \times r$ upper left block of $Y[U_{v}]$.
  By using the notation and the result of Shimura \cite{Sh}, (7.12),
  we have
  \begin{align*}
    \xi(\eta, H_{v}', k, 0) &=
                              \int_{S_{r}(\mathbb{R})}\left|\text{det} (x + \eta i)\right|_{\infty}^{-k}
                              \textbf{e}_{\infty}(-\text{tr}H_{\infty}'x)
                              d\mu_{\infty}(x),\\
                            &=
                              2^{(1-r)r}i^{-rk}(2\pi)^{rk}
                              \Gamma_{r}(k)^{-1}\left(\text{det} H_{v}'\right)^{k-r}
                              \textbf{e}(i \text{tr} H_{v}'\eta).
  \end{align*}
  Here
  \begin{equation*}
    \Gamma_{m}(s) = \pi^{m(m-1)/2}
    \prod_{i=0}^{m-1}\Gamma(s-i).
  \end{equation*}

  In the rest of the proof, we assume that $r$ is even for simplicity.
  We omit the proof for an odd $r$ since the proof is the same.
  We set
  \begin{equation*}
    L_{r}(k) = \prod_{q: \text{ prime}}L_{r, q}(k)
    = \prod_{i=0}^{r-1}L(k - i, \chi_{\boldsymbol{K}}^{i})^{-1},
  \end{equation*}
  and set $r = 2r'$ with $r' \in \mathbb{Z}_{\geq 1}$.
  Then by functional equations of Dirichlet $L$-functions, we have
  \begin{align*}
    L_{r}(k) = (-1)^{r'}2^{r} &
                                \left(2\pi\right)^{r(r-1)/2 - kr}
    \\ & \times
         D_{\boldsymbol{K}}^{(k+1/2)r' - r'(r'+1)}
         \prod_{i=0}^{r-1}\Gamma(k-i)\prod_{i=1}^{r}L(i-k, \chi_{\boldsymbol{K}}^{i-1})^{-1}.
  \end{align*}
  Thus we have
  \begin{align*}
    c_{\mu}L_{r}(k)\xi(\eta, H_{v}', k, 0)
    =
    (-1)^{r'}2^{r} & D_{\boldsymbol{K}}^{r'(k - r)}\left( \text{det} H'_{\infty} \right)^{k-r}
    \\ & \times
         \prod_{i=1}^{r}L(i-k, \chi_{\boldsymbol{K}}^{i-1})^{-1}
         \textbf{e}(i\text{tr} HY).
  \end{align*}
  Let $H \in \Lambda_{r}(\mathcal{O}_{\boldsymbol{K},q})$ with $\det H \ne 0$.
  By Theorem \ref{Ikeda}, we have
  \begin{equation*}
    \mathcal{F}_q(H, q^{-k}) = \left|\gamma(H)\right|_{q}^{k-r}\underline{\chi}_{\boldsymbol{K}, q}(\gamma(H))
    \mathcal{F}_{q}(H, q^{k-2r}).
  \end{equation*}
  Therefore, we have
  \begin{align*}
    & c_{\mu}L_{r}(k)\xi(\eta, H_{v}', k, 0)
      \prod_{q: \text{prime}} \mathcal{F}_{q}(H_{q}', q^{-k})
      = \\
    & 2^{r}\prod_{v: \text{place of } \mathbb{Q}}\left|\gamma(H_{v}')\right|_{v}^{r-k}
      \prod_{q: \text{prime}} \mathcal{F}_{q}(H_{q}', q^{k-2r})
      \prod_{i=1}^{r}L(i-k, \chi_{\boldsymbol{K}}^{i-1})^{-1}
      \textbf{e}(i\text{tr} HY).
  \end{align*}
  Since $H_{v}'[\alpha_{v}^{-1}] = H' \in S_{r}(\mathbb{Q})$, we have the assertion of Theorem \ref{thm:explicitH}.
%\end{proof}
\\
\\

\label{sec:proof-theor-refthm:1}

Acknowledgement: We thank S.~B\"{o}cherer
for helpful discussions related to this work.
This work was supported by JSPS KAKENHI: first author, Grant-in-Aid (C)
(No. 25400031); second author, Grant-in-Aid (B) (No. 16H03919).
%%%%%%%%%%%%%%%%%%%%%%%%%%%%%%%%%%%%%%%%%%%%%%%%%%%%%%

\noindent
S.~Nagaoka\\
Dept. Mathematics, Kindai Univ., Higashi-Osaka\\
Osaka 577-8502,Japan\\
nagaoka@math.kindai.ac.jp
\\
\\
S.~Takemori\\
Max-Planck-Institut f\"{u}r Mathematik\\
Vivatsgasse 7, 53111 Bonn, Germany\\
stakemori@gmail.com

\end{document}